\documentclass[11pt,a4paper]{amsart}
\usepackage{amssymb, amsmath, amsthm}

\usepackage{pstricks}
\usepackage{graphicx}
\usepackage{mathrsfs}
\usepackage{extarrows}
\usepackage{amsthm}
\usepackage{amssymb}
\usepackage{xypic,enumerate, comment, enumitem}
\usepackage[all,cmtip]{xy}
\usepackage{fancyvrb}
\usepackage{longtable}
\usepackage{array}
\newcolumntype{L}{>{$}l<{$}} 
\allowdisplaybreaks

\usepackage{hyperref}
\hypersetup{
    colorlinks=true,       
    linkcolor=blue,        
    citecolor=blue,        
    filecolor=blue,        
    urlcolor=blue          
}

\theoremstyle{plain}

\newtheorem{theorem}{Theorem}[section]
\newtheorem{proposition}[theorem]{Proposition}
\newtheorem{corollary}[theorem]{Corollary}
\newtheorem{lemma}[theorem]{Lemma}
\newtheorem{conjecture}[theorem]{Conjecture}
\newtheorem*{theorem*}{Theorem}
\newtheorem*{proposition*}{Proposition}
\newtheorem*{corollary*}{Corollary}
\newtheorem*{lemma*}{Lemma}
\newtheorem*{conjecture*}{Conjecture}
\newtheorem*{thmA}{Theorem A}
\newtheorem*{thmB}{Theorem B}

\theoremstyle{definition}
\newtheorem{definition}[theorem]{Definition}
\newtheorem{example}[theorem]{Example}

\newtheorem*{definition*}{Definition}
\newtheorem*{example*}{Example}

\theoremstyle{remark}
\newtheorem{remark}[theorem]{Remark}
\newtheorem*{remark*}{Remark}

\newcommand{\HH}{\operatorname{H}}

\newcommand{\G}{\mathrm{G}}
\newcommand{\Fl}{\mathrm{Fl}}

\newcommand{\bQ}{\mathbb{Q}}

\newcommand{\bP}{\mathbb{P}}

\newcommand{\cO}{\mathcal{O}}

\newcommand{\cU}{\mathcal{U}}

\newcommand{\cQ}{\mathcal{Q}}

\newcommand{\tV}{\tilde{V}}

\newcommand{\Pic}{\operatorname{Pic}}

\newcommand{\Gr}{\operatorname{G}}

\newcommand{\Bl}{\operatorname{Bl}}

\renewcommand{\implies}{\Rightarrow}

\renewcommand\bar\overline
\renewcommand\tilde\widetilde

\newcommand{\Sym}{\operatorname{Sym}}
\renewcommand{\geq}{\geqslant}
\renewcommand{\leq}{\leqslant}

\newcommand{\PP}{\bP}
\newcommand{\OO}{\mathcal{O}}

\newcommand{\QQ}{\mathbf{Q}}
\newcommand{\ZZ}{\mathbf{Z}}
\newcommand{\CC}{\mathbf{C}}

\newcommand{\irr}{\operatorname{irr}}

\title{On some invariants of cubic fourfolds}
\author{Frank Gounelas}
\address{Georg-August-Universit\"at G\"ottingen, Fakult\"at f\"ur Mathematik und Informatik, Bunsenstr. 3-5, 37073 G\"ottingen, Germany}
\email{gounelas@mathematik.uni-goettingen.de}

\author{Alexis Kouvidakis}
\address{Dept. of Mathematics and Applied Mathematics, University of Crete, 70013 Heraklion, Greece.}
\email{kouvid@uoc.gr}

\date{\today}

\begin{document}

\maketitle
\begin{abstract}
For a general cubic fourfold $X\subset\PP^5$ with Fano variety $F$, we compute the Hodge numbers of the locus $S\subset F$ of lines of
second type and the class of the locus $V\subset F$ of triple lines, using the description of the latter in terms of flag varieties.
We also give an upper bound of 6 for the degree of irrationality of the Fano scheme of lines of any smooth cubic
hypersurface.
\end{abstract}

\setcounter{tocdepth}{1}
\tableofcontents

\section{Introduction}\label{sec:introduction}

Let $X\subset\PP^5_\CC$ be a general cubic fourfold and $F=F(X)\subset\G(2,6)$ its Fano scheme of lines, which is a
four-dimensional hyperk\"ahler variety. The normal bundle of a line $\ell\in F$ decomposes as one of the following two
\begin{align*}
    N_{\ell/X}\cong\OO(1)\oplus\OO^2\text{, or } \OO(1)^2\oplus\OO(-1)
\end{align*}
and $\ell$ is called of first or second type respectively. The locus of
second type lines is a smooth projective irreducible surface $S\subset F$ which has drawn considerable interest since
the landmark paper \cite{cg} of Clemens--Griffiths. The aim of this paper is to study some invariants of $S$ and $F$.

In Section \ref{sec:background} we summarise what is known about $S$ and $F$ and fix notation. In Section
\ref{sec:second type surface} we use Amerik's description of the second type locus $S$ as the degeneracy locus of the
universal Gauss map 
\begin{align*}
f:\Sym^2\cU_F&\to\cQ^\vee_F\\
S=D_2(f)&\subset F,
\end{align*}
the Harris--Tu formula as well as Borel--Bott--Weil computations on the Grassmannian from Section \ref{sec:grass} to
compute the Hodge numbers of $S$.

\begin{thmA}
    If $X\subset\PP^5$ is a general cubic then the second type locus $S\subset F$ is a smooth irreducible surface whose
    Hodge numbers are as follows
    \begin{align*}
        h^{1,0}&=q=0, \\
        h^{2,0}&=p_g=449, \\
        h^{1,1}&=1665,
    \end{align*}
    whereas $\pi_1(S,s)$ contains a non-trivial element of order 2.
\end{thmA}

The order of the torsion element in the above theorem was pointed out to us by Huybrechts (see Remark
\ref{rem:2-torsion}), who also independently calculated the above invariants in his lecture notes on cubic
hypersurfaces, although our approach using Borel--Bott--Weil directly on $F$ leads to a more refined analysis of the
projective embedding $S$ in the Pl\"ucker space. 

In the final Section \ref{sec:irr} we extend results from \cite{irrgk} to prove the following
\begin{thmB}
Let $X\subset\PP^{n+1}$ be a smooth cubic hypersurface and $F(X)$ its Fano scheme 
of lines. Then degree of irrationality of $F(X)$, i.e., the minimal degree of a dominant, generically finite, rational
map to $\PP^{2(n-2)}$, satisfies
\[\irr(F(X))\leq 6.\]
\end{thmB}

\textbf{Acknowledgements} We would like to thank Olof Bergvall, Daniel Huybrechts and Fabio Tanturri for helpful
correspondence, and Alexander Kuzne\-tsov for suggesting the construction in Section \ref{sec:V}. The first author was
partially supported by the ERC Consolidator Grant 681838 ``K3CRYSTAL''.

\section{Background and notation}\label{sec:background} 

As the notation surrounding cubic fourfolds is substantial, we devote this section to fixing that used in the
paper and recalling some basic properties, so that it acts as a reference for later sections.

For a vector bundle $E$ we denote by $\PP(E)=\mathrm{Proj}(\mathrm{Sym}(E^\vee))$, so that projective space parametrises one
dimensional subspaces. We denote by $\G(k,n)$ the space of $k$-dimensional subspaces of $\CC^n$, with universal bundle
$\cU$ of rank $k$ and universal quotient bundle $\cQ$ of rank $n-k$. We will denote by $\sigma_I$ the standard Schubert
cycles for an index $I$ so that, e.g., $\sigma_i={\mathrm c}_i(\cQ)$ for $i\geq1$.

Throughout, $X\subset\PP^5$ will be a smooth cubic fourfold with $H_X=\OO_{X}(1)$ and $F\subset \G(2,6)$ the Fano
scheme of lines contained in $X$ which is a hyperk\"ahler fourfold \cite{bd}. To unburden notation, we will often be
sloppy in distinguishing a line $\ell\subset X$ and the point $[\ell]\in F$ that it defines. We denote by $\cU_F,\cQ_F$
the restrictions of $\cU,\cQ$ to $F$. 

The subvariety $F\subset \G(2,6)$ is given by a section of the rank four bundle $\Sym^3\cU^*\cong q_*p^*\OO_{\PP^5}(3)$
where $p,q$ are the projections from the universal family $\G(2,6)\leftarrow I\to \PP^5$. In fact it is the section
induced, under this isomorphism, by $f\in k[x_0,\ldots,x_5]_3$ whose vanishing is $X$ (see \cite[Proposition 6.4]{3264})
and its cohomology class in the Grassmannian is given by ${\mathrm c}_4(\Sym^3\cU^*)$ which can be computed as follows
(see \cite[Example 14.7.13]{fulton})
\begin{align}\label{eq:class of F}
\begin{split}
    [F] &= 18{\mathrm c}_1(\cU^*)^2{\mathrm c}_2(\cU^*) + 9{\mathrm c}_2(\cU^*)^2 \\
        &= 18\sigma_1^2\sigma_{1,1}+9\sigma_{1,1}^2 \\
        &= 27\sigma_2^2-9\sigma_1\sigma_3-18\sigma_4.
\end{split}
\end{align}

Following \cite{cg}, there are two types of lines $\ell\in F$, depending on the decomposition of the normal bundle
$N_{\ell/X}$.

\begin{definition}
    We say that a line $\ell\subset X$ is 
    \begin{enumerate}
        \item of \textit{first type} if $N_{\ell/X}\cong\OO(1)\oplus\OO^2$,
        \item of \textit{second type} if $N_{\ell/X}\cong\OO(1)^2\oplus\OO(-1)$.
    \end{enumerate}
\end{definition}
An equivalent geometric description is as follows: $\ell$ is of
\begin{enumerate}
    \item first type if there is a unique $\Pi_\ell=\PP^2$ tangent to $X$ along $\ell$.
    \item second type if there is a family $\Pi_{\ell,t}=\PP^2$, $t\in\PP^1$, of $2$-planes tangent
    to $X$ along $\ell$.
\end{enumerate}
Denote by 
\begin{align*}
S&:=\{\ell : \ell\text{ is of second type}\}\subset F
\end{align*}
the \textit{locus of second type lines}. 

Denote by $H_F={\mathrm c}_1(\cU^\vee_F)$ the Pl\"ucker ample line bundle on $F$ and by $H_S$ the restriction on $S$.
The following is a combination of \cite[Lemma 1]{amerik}, \cite[\S 3]{osy}, \cite[Proposition 6.4.9]{huybrechts}.
\begin{theorem} \label{thm:amerik and osy}
    If $X\subset\PP^5$ is a cubic fourfold then $S$ is 2-dimensional and is the degeneracy locus of the Gauss map, i.e.,
    the following morphism of vector bundles
    \[ \Sym^2\cU_F \to \cQ_F^\vee.\]
    In particular ${\mathrm c}_1(K_S)=3H_S$ in $\HH^2(S,\bQ)$ and the class of $S$ in $\mathrm{CH}^2(F)$ is given by
    \[[S]=5({\mathrm c}_1(\cU_F^\vee)^2-{\mathrm c}_2(\cU_F^\vee))=5{\mathrm c}_2(\cQ_F)=5\sigma_2|_F.\]
    If $X$ is general, $S$ is a smooth projective irreducible surface.
\end{theorem}

We motivate now the study of $S$. Consider the \textit{Voisin map} of \cite{voisinintrinsic} $\phi:F\dashrightarrow
F, \ell\mapsto\ell'$, taking a general line $\ell$ and giving the residual line $\ell'$ in the tangent 2-plane
$\Pi_\ell$ to $\ell$, i.e., $\Pi_\ell\cap X=2\ell+\ell'$. Note that this is not defined on $S$ nor on any lines
contained in a plane contained inside $X$. Containing a plane is a divisorial condition, so for $X$ outside this locus,
we can resolve this map with one blowup $\tilde{F}=\operatorname{Bl}_SF$ along the surface $S$. The map $\phi$ has been
used in various contexts (see, e.g., \cite{amerikvoisin, svfourier}), so it is important to understand its locus of
indeterminacy. See also \cite[\S 2, \S 6]{huybrechts} for further references and motivation.

As another example, \cite[Theorem 0.2]{osy} prove that if $X$ is very general then for every rational curve $C\in F$ of
class $\beta$, the generator of $\HH_2(X,\ZZ)^{\mathrm{alg}}$, there exists a unique $s\in S$ so that
$C=\phi(q^{-1}(s))$. In \cite{no} this is used to count the number of arithmetic genus $1$ curves of fixed general
$j$-invariant in $F$ of class $\beta$, and in \cite{gk21} to count the number of nodal rational curves of class $\beta$
respectively.

\section{Cohomology of $\G(2,6)$}\label{sec:grass}

This section contains some ancillary computations necessary for the next section. We briefly recall the necessary
notation for the Borel--Weil--Bott Theorem used to compute various cohomology groups of tautological bundles on the
Grassmannian $\G(2,6)$ with universal sub and quotient bundle $\cU,\cQ$ respectively. For a quick introduction we found
\cite[Appendix A]{bcp} and \cite{mukai} helpful, although a more thorough reference is \cite{weyman}. 

Denote by $\rho=(6,5,4,3,2,1)$, $w=(w',w'')\in\ZZ^2\oplus\ZZ^4$ respectively and $\Sigma_w$ the standard
Weyl module. If $w+\rho$ is \textit{regular}, i.e., all its components are distinct integers, then the BWB Theorem
states that
\[
\HH^{\ell(w)}(\G(2,6), \Sigma_{w'}\cU^* \otimes \Sigma_{w''}\cQ^*) \cong \Sigma_{\sigma(w+\rho)-\rho}\CC^6
\]
is the only non-trivial cohomology group of this sheaf. In the above, $\sigma$ is the unique element of the symmetric group
$S_6$ which permutes the components of $w+\rho$ so that they are non-increasing, i.e.,
$\sigma(w+\rho)=(\lambda_1,\ldots,\lambda_6)$ with $\lambda_1\geq\ldots\geq\lambda_6$, and $\ell(w)$ is defined as the
length of $\sigma$ in the sense of the number of transpositions of the form $(i\hspace{5pt} i+1)$ that $\sigma$
constitutes of. If on the other hand $w+\rho$ is not regular, then all cohomology groups are zero. 

We recall the formula, e.g., from \cite[Theorem 6.3]{fultonharris}, that if $\lambda=(\lambda_1,\ldots,\lambda_6)$
is such that $\lambda_1\geq\ldots\geq\lambda_6\geq1$ then
\[
\dim \Sigma_\lambda\CC^6 = \prod_{1\leq i<j\leq6} \frac{\lambda_i-\lambda_j+j-i}{j-i},
\]
whereas for an arbitrary non-increasing sequence $\lambda$, we may twist by some large weight (e.g.,
$(|\lambda_6|+1,\ldots,|\lambda_6|+1)$) to make all components positive - this has the effect of tensoring the
representation by a 1-dimensional one which does not change the dimension.

The first task is to decompose various tautological sheaves into irreducible representations. Here are some
examples of irreducible representations
\begin{align*}
    \Sigma_{(1,1)}\cU^* &\cong H \\
    \Sigma_{(1,-1)}\cU^* &\cong (\Sym^2\cU)(H) \\
    \Sigma_{(0,-1,-1,-1)}\cQ^* &\cong \cQ^*(1).
\end{align*}

\begin{proposition}\label{prop:coh grass}
    The non-zero cohomology groups of $\wedge^p\Sym^3\cU \otimes (\Sym^2\cU)(tH)$ and
    $\wedge^p \Sym^3\cU \otimes \cQ^*(tH)$ on $\G(2,6)$ for $t=1$ are 
    \begin{align*}
        \HH^4(\G(2,6), \wedge^2\Sym^3\cU \otimes \Sym^2\cU(H)) &\cong \CC^{36} \\
        \HH^0(\G(2,6), \cQ^*(H)) &\cong \CC^{20} \\
        \HH^5(\G(2,6), \wedge^3 \Sym^3\cU \otimes \cQ^*(H)) &\cong \CC
    \end{align*}
    whereas for $t=-2$ they are
    \begin{align*}
        \HH^8(\G(2,6), \wedge^3\Sym^3\cU \otimes \Sym^2\cU(-2H)) &\cong \CC^{126} \\
        \HH^8(\G(2,6), \wedge^4\Sym^3\cU \otimes \Sym^2\cU(-2H)) &\cong \CC^{1134} \\
        \HH^5(\G(2,6), \Sym^3\cU \otimes \cQ^*(-2H)) &\cong \CC\\
        \HH^8(\G(2,6), \wedge^4 \Sym^3\cU \otimes \cQ^*(-2H)) &\cong \CC^{560}.
    \end{align*}
\end{proposition}
\begin{proof}
    Using the following code in the SchurRing package of Macaulay2, 
    \begin{small}
    \begin{verbatim}
loadPackage "SchurRings";
S = schurRing(QQ,s,2);
for i from 0 to 4 do (
print (exteriorPower(i,symmetricPower(3,s_1))*symmetricPower(2,s_1));)
    \end{verbatim}
    \end{small}
    we compute the weights of the irreducible components of the representation $\wedge^p\Sym^3\cU \otimes \Sym^2\cU(H)$
    as follows
    \begin{small}
    \begin{center}
    \begin{tabular}{|L|L|L|L|}
        \hline
        p & w' & w+\rho=(w';0,0,0,0)+\rho & \ell(w) \\ 
        \hline
        0 & (1,-1) & (7,4,4,3,2,1) & -1 \\
        \hline
        1 & (1,-4)\oplus (0,-3) & (7,1,4,3,2,1)\oplus(6,2,4,3,2,1) & -1\oplus-1\\
          & \oplus (-1,-2) & \oplus(5,3,4,3,2,1) & \oplus-1 \\
        \hline
        2 & (0,-6)\oplus(-1,-5) & (6,-1,4,3,2,1)\oplus(5,0,4,3,2,1) & 4\oplus4\\
          & \oplus(-2,-4)^{\oplus2} & \oplus(4,1,4,3,2,1)^{\oplus2} & \oplus-1\oplus-1 \\
        \hline
        3 & (-2,-7)\oplus(-3,-6) &(4,-2,4,3,2,1)\oplus(3,-1,4,3,2,1) & -1\oplus-1\\
          & \oplus(-4,-5) & \oplus(2,0,4,3,2,1) & \oplus-1 \\
        \hline
        4 & (-5,-7) & (1,-2,4,3,2,1) & -1 \\
        \hline
    \end{tabular}
    \end{center}
    \end{small}
    since for example a decomposition into irreducibles for $p=2$ is
    \[
    \wedge^2\Sym^3\cU \otimes \Sym^2\cU(H) \cong \Sigma_{(0,-6)}\cU^* \oplus \Sigma_{(-1,-5)}\cU^* \oplus
    (\Sigma_{(-2,-4)}\cU^*)^{\oplus 2}.
    \]
    In the table, $\ell(w)=-1$ signifies that the weight $w$ is not regular. From the Borel--Weil--Bott Theorem, we obtain
    \[
        \HH^i(\G(2,6), \wedge^p\Sym^3\cU \otimes \Sym^2\cU(H))=0\text{ for }p=0,1,3,4\text{ and }i\geq0.
    \]
    For $p=2$, as 
    \begin{align*}
    \sigma((6,-1,4,3,2,1)+\rho)-\rho&=(0,-1,-1,-1,-1,-2),\\
    \sigma((5,0,4,3,2,1)+\rho)-\rho&=(-1,-1,-1,-1,-1,-1)
    \end{align*}
    we obtain
    \begin{align*}
    \HH^4(\G(2,6), \wedge^2\Sym^3\cU \otimes \Sym^2\cU(H)) \cong& \Sigma_{(0,-1,\ldots,-1,-2)}\CC^6\oplus
    \Sigma_{(-1,\ldots,-1)}\CC^6\\ 
    \cong& \CC^{35}\oplus\CC \cong \CC^{36}.
    \end{align*}
    Similarly, the table for $\wedge^p \Sym^3\cU \otimes \cQ^*(H)$ is as follows
    \begin{small}
    \begin{center}
    \begin{tabular}{|L|L|L|L|}
        \hline
        p & w' & w+\rho=(w',0,-1,-1,-1)+\rho & \ell(w) \\ 
        \hline
        0 & (0,0) & (6,5,4,2,1,0) & 0 \\
        \hline
        1 & (0,-3)& (6,2,4,2,1,0) & -1 \\
        \hline
        2 & (-1,-5)\oplus(-3,-3) & (5,0,4,2,1,0)\oplus(3,2,4,2,1,0) & -1\oplus-1\\
        \hline
        3 & (-3,-6) & (3,-1,4,2,1,0) & 5\\
        \hline
        4 & (-6,-6) & (0,-1,4,2,1,0) & -1 \\
        \hline
    \end{tabular}
    \end{center}
    \end{small}
    so the only non-zero cohomology groups occur for $p=0,3$. Using the same formulas as above we compute
    \begin{align*}
        \HH^0(\G(2,6), \cQ^*(H)) \cong& \Sigma_{(0,0,0,-1,-1,-1)}\CC^6 \cong \CC^{20}\\
        \HH^5(\G(2,6), \wedge^3 \Sym^3\cU \otimes \cQ^*(H)) \cong& \Sigma_{(-2,-2,-2,-2,-2,-2)}\CC\cong\CC.
    \end{align*}
    The table for $\wedge^p\Sym^3\cU\otimes(\Sym^2\cU)(-2H)$ is as follows.
    \begin{small}
    \begin{center}
    \begin{tabular}{|L|L|L|L|}
        \hline
        p & w' & w+\rho=(w'+(6,5),4,3,2,1) & \ell(w) \\ 
        \hline
        0 & (-2,-4) & (4,1,4,3,2,1) & -1 \\
        \hline
        1 & (-2,-7)\oplus (-3,-6) & (4,-2,4,3,2,1)\oplus(3,-1,4,3,2,1) & -1\oplus-1\\
          & \oplus (-4,-5) & \oplus(2,0,4,3,2,1) & \oplus-1 \\
        \hline
        2 & (-3,-9)\oplus(-4,-8) & (3,-4,4,3,2,1)\oplus(2,-3,4,3,2,1) & -1\oplus-1\\
          & \oplus(-5,-7)^{\oplus2} & \oplus(1,-2,4,3,2,1)^{\oplus2} & \oplus-1\oplus-1 \\
        \hline
        3 & (-5,-10)\oplus(-6,-9) &(1,-5,4,3,2,1)\oplus(0,-4,4,3,2,1) & -1\oplus8\\
          & \oplus(-7,-8) & \oplus(-1,-3,4,3,2,1) & \oplus8 \\
        \hline
        4 & (-8,-10) & (-2,-5,4,3,2,1) & 8 \\
        \hline
    \end{tabular}
    \end{center}
    \end{small}
    giving
    \begin{align*}
        \HH^8(\wedge^3\Sym^3\cU\otimes(\Sym^2\cU)(-2H)) \cong&
        \Sigma_{(-2,\ldots,-2,-5)}\CC^6\oplus\Sigma_{(-2,\ldots,-2,-3,-4)}\CC^6 \\
         \cong& \CC^{56}\oplus\CC^{70} \cong \CC^{126} \\
        \HH^8(\wedge^4\Sym^3\cU\otimes(\Sym^2\cU)(-2H)) \cong& \Sigma_{(-2,-2,-2,-2,-4,-6)}\CC^6 \cong \CC^{1134}.
    \end{align*}
    Similarly, the table for $\wedge^p \Sym^3\cU \otimes \cQ^*(-2H)$ is as follows, noting that
    $\cQ^*(-2H)\cong\Sigma_{(3,2,2,2)}\cQ^*$.
    \begin{small}
    \begin{center}
    \begin{tabular}{|L|L|L|L|}
        \hline
        p & w' & w+\delta=(w',3,2,2,2)+\rho & \ell(w) \\ 
        \hline
        0 & (0,0) & (6,5,7,5,4,3) & -1 \\
        \hline
        1 & (0,-3)& (6,2,7,5,4,3) & 5 \\
        \hline
        2 & (-1,-5)\oplus(-3,-3) & (5,0,7,5,4,3)\oplus(3,2,7,5,4,3) & -1\oplus-1\\
        \hline
        3 & (-3,-6) & (3,-1,7,5,4,3) & -1\\
        \hline
        4 & (-6,-6) & (0,-1,7,5,4,3) & 8 \\
        \hline
    \end{tabular}
    \end{center}
    \end{small}
    giving 
    \begin{align*}
        \HH^5(\Sym^3\cU \otimes \cQ^*(-2H)) &\cong \Sigma_{(1,1,1,1,1,1)}\CC^6\cong\CC\\
        \HH^8(\wedge^4\Sym^3\cU \otimes \cQ^*(-2H)) &\cong \Sigma_{(1,0,0,0,-2,-2)}\CC^6\cong\CC^{560}.\qedhere
    \end{align*}
\end{proof}

\section{Hodge numbers of $S$}\label{sec:second type surface} 

In Theorem \ref{thm:amerik and osy}, we described how $S$ is given as the degeneracy locus of the map
\begin{align*}
    f:\Sym^2\cU_F\to \cQ_F^\vee.
\end{align*}
Restricting to $S$ we thus have the following sequence of vector bundles 
\begin{align}\label{seq:amerik on S}
0\to K\to \Sym^2\cU_S\xrightarrow{f|_S} \cQ_S^\vee\to C\to 0
\end{align}
where $K$ is a line bundle and $C$ of rank 2. Note that there is a formula for the normal bundle of a degeneracy locus
in \cite[\S 3]{harristu} giving
\[N_{S/F} = K^\vee\otimes C.\]

The map $f$ is generically injective when considered on $F$, hence injective and Amerik \cite[\S 2]{amerik} has
constructed the following resolution of the ideal sheaf $I_S$ of $S\subset F$
\begin{align}\label{eq:resolution IS}
    0\to \Sym^2\cU_F(-2H)\to&\cQ^\vee_F(-2H)\to I_S\to0.
\end{align}
A short explanation is in order concerning the above. The cokernel of $f$ is torsion-free by noting that the
degeneracy locus $S$ does not have any divisorial components (see the local computations of \cite[p.32-33]{friedman}).
From this one obtains $\operatorname{coker}(f)=M\otimes I_S$ for some line bundle $M$, and an Euler characteristic
computation in \cite{amerik} gives $M=2H$.

\begin{proposition}
    For $S$ the surface parametrising lines of second type on a cubic fourfold $X$ we have
    \begin{enumerate}
        \item $K_S^2=2835$,
        \item $\chi(\OO_S)=450$.
    \end{enumerate}
\end{proposition}
\begin{proof}
    As ${\mathrm c}_1(K_S)=3H_S \in \HH^2(S, \QQ)$ and $H_S^2=315$ from Theorem \ref{thm:amerik and osy}, we compute
    that $K_S^2=2835$.  To simplify notation for this proof we denote by
    \begin{align*}
        \mathcal{E}&=\cQ_F^\vee,\\
        \mathcal{F}&=\Sym^2\cU_F.
    \end{align*}
    To compute $\chi(\OO_S)$ we compute first the Chern numbers of $K$ and $C$. For this we use the Harris--Tu formula
    \cite{harristu}, although we follow the notation of \cite{pragacz}. We denote the Segre polynomial 
    \[{\mathrm s}_t(\mathcal{E}-\mathcal{F}):=\sum {\mathrm s}_k(\mathcal{E}-\mathcal{F})t^k :=
    {\mathrm s}_t(\mathcal{E}){\mathrm c}_t(\mathcal{F})\] 
    where ${\mathrm s}_t(\mathcal{E}), {\mathrm c}_t(\mathcal{F})$ are the Segre and Chern polynomials of $\mathcal{E}$ and
    $\mathcal{F}$ respectively. Written in terms of the standard Schubert cycles $\sigma_i:={\mathrm c}_i(\cQ)$ on $\G(2,6)$
    we have
    \begin{align*}
        \sum {\mathrm s}_k(\mathcal{E}-\mathcal{F})t^k =&\hspace{3pt} 1 - 2\sigma_1t + (4\sigma_1^2-5\sigma_2)t^2 + (\sigma_1\sigma_2+\sigma_3)t^3 + \\
                         & \hspace{3pt}(2\sigma_2^2-4\sigma_1\sigma_3+2\sigma_4)t^4 +
                         (-4\sigma_2\sigma_3+4\sigma_1\sigma_4)t^5,
    \end{align*}
    and in what follows we denote by ${\mathrm s}_i:={\mathrm s}_i(\mathcal{E}-\mathcal{F})$. For a partition $I=(i_1,i_2,\ldots)$ we denote by
    \[{\mathrm s}_I(\mathcal{E}-\mathcal{F}):=\det\left[({\mathrm s}_{i_p-p+q})_{p,q}\right]\] so now \cite[Example 5.4]{pragacz} (note there are some typos fixed in
    a later paper) gives the following intersection numbers, all taking place on $F$, i.e., intersected with $[F]$ from
    \eqref{eq:class of F}
    \begin{align*}
    {\mathrm c}_2(C) & = {\mathrm s}_{(2,2)}(\mathcal{E}-\mathcal{F}) = ({\mathrm s}_2^2-{\mathrm s}_1{\mathrm s}_3) = 495, \\
    {\mathrm c}_1^2(C) & = ({\mathrm s}_{(3,1)}+{\mathrm s}_{(2,2)}) = ({\mathrm s}_1{\mathrm s}_3-{\mathrm s}_4) = -180, \\
    {\mathrm c}_1^2(K) & = ({\mathrm s}_{(1,1,2)}+{\mathrm s}_{(1,1,1,1)}) = ({\mathrm s}_1^4-3{\mathrm s}_1^2{\mathrm s}_2+2{\mathrm s}_1{\mathrm s}_3+{\mathrm s}_2^2-{\mathrm s}_4) = 315.
    \end{align*}
    From the tangent sequence of $S\subset F$ and the fact that $K_F=0$ we obtain \[3H_S={\mathrm c}_1(K_S)={\mathrm
    c}_1(N)=-2{\mathrm c}_1(K)+{\mathrm c}_1(C)\] from which ${\mathrm c}_1(K){\mathrm c}_1(C)=-315$ and hence ${\mathrm
    c}_2(N_{S/F})=1125$. On the other hand from the tangent sequence of $F\subset \G(2,6)$ 
    \[
        0\to T_F\to \cQ\otimes\cU^\vee \to \Sym^3\cU^\vee\to0
    \]
    we have ${\mathrm c}_2(T_F) = -3\sigma_1^2|_F+8\sigma_2|_F$,
    giving \[{\mathrm c}_2(T_S) = {\mathrm c}_2(T_F)[S]-{\mathrm c}_2(N_{S/F})-{\mathrm c}_1(T_S){\mathrm c}_1(N_{S/F})=2565.\] From the Noether formula we compute now
    \[\chi(\OO_S) = \frac{1}{12}({\mathrm c}_1(T_S)^2+{\mathrm c}_2(T_S))=450.\qedhere\]
\end{proof}
\begin{remark}
Using the fact that $S$ is isomorphic to $S'$ a section of the vector bundle
$\mathcal{E}=\pi^*\cQ_F^\vee\otimes\OO_{\PP(\Sym^2 \cU_F)}(1)$ on $\pi:\PP(\Sym^2 \cU_F)\to F$, we have from
\cite[p.54]{fp} the formula \[\chi_{\mathrm top}(S) = \int_S {\mathrm c}_{\mathrm top}(\mathcal{E}){\mathrm
c}(\mathcal{E})^{-1}{\mathrm c}(\PP(\Sym^2 \cU_F))\] which can also be used to compute $\chi(\OO_S)$. In fact, recently
Huybrechts \cite[Proposition 6.4.9]{huybrechts} has studied the ideal sheaf $I_{S'}$, proving that Sequence
\eqref{seq:amerik on S} on $S$ is 
\begin{align}\label{seq:huybrechts}
0\to L\to \Sym^2\cU_S\to \cQ_S^\vee\to N_{S/F}\otimes L\to 0
\end{align}
for a line bundle $L$ satisfying $-2L=2H_S$. From this one can, by taking Euler characteristics, also obtain that
$\chi(\OO_S)=450$. Studying cohomological vanishing on $\PP(\Sym^2\cU_F)$ he also obtains $h^1(S, \OO_S)=0$ like we do
in what follows.
\end{remark}

Our aim now is to compute $q=h^1(S,\OO_S)$ or $p_g$, noting that \[\chi(S, \OO_S)=1-q+p_g\] so one determines the other
from the above computation. This will be achieved by computing cohomology from Sequence \eqref{eq:resolution IS}.
As $F$ is the vanishing of a section of $\Sym^3\cU^\vee$, we can consider the Koszul resolution 
\begin{align}\label{koszul}
    0\to\wedge^4 \Sym^3\cU\to\cdots\to \Sym^3\cU\to\OO_{\G(2,6)}\to\OO_F\to0
\end{align}
from which it becomes clear that in order to compute groups such as \[\HH^i(F, \Sym^2\cU_F(H))\] we will need to compute
the groups \[\HH^i(\G(2,6), \wedge^p\Sym^3\cU\otimes(\Sym^2\cU)(H)),\] which was achieved using the Borel--Weil--Bott
Theorem in Section \ref{sec:grass}.

\begin{theorem}
    The Hodge numbers of $S$ are as follows
    \begin{align*}
        h^{1,0}&=q=0, \\
        h^{2,0}&=p_g=449,\\
        h^{1,1}&=1665.
    \end{align*}
    Also, $\Pic S \cong \mathrm{NS}(S)$ and $\Pic^\tau S\neq0$, i.e., $S$ has torsion in the N\'eron--Severi group and has
    non-trivial fundamental group.
\end{theorem}
\begin{proof}
    Tensoring Sequence \eqref{koszul} with $\Sym^2\cU(tH)$ and $\cQ^\vee(tH)$ and using the hypercohomology
    spectral sequence \cite[B.1.5]{lazarsfeldv1} we obtain the following second quadrant spectral sequences
    \begin{align*}
        E_1^{p,q}&=\HH^q(\wedge^{-p}\Sym^3\cU \otimes (\Sym^2\cU)(tH)) \implies \HH^{p+q}(F, \Sym^2\cU_F(tH)),\\
        E_1^{p,q}&=\HH^q(\wedge^{-p}\Sym^3\cU\otimes\cQ^\vee(tH)) \implies \HH^{p+q}(F, \cQ^\vee_F(tH)).
    \end{align*}
    From Proposition \ref{prop:coh grass} for $t=-2$ and the first spectral sequence, we have that
    \[d_{-4,8}:E_1^{-4,8}\to E_1^{-3,8}\]
    is the only non-trivial differential between the only two non-trivial terms of the $E_1$-page. Since $\HH^5(F,
    (\Sym^2\cU_F)(-2H))=0$ as $\dim F=4$, it must be that $E_\infty^{-3,8}=0$ and so that $d_{-4,8}$
    is surjective. This gives that $E_\infty^{-4,8}=E_2^{-4,8}\cong\CC^{1008}$ and hence that
    $\HH^4((\Sym^2\cU_F)(-2H))=\CC^{1008}$ is the only non-zero cohomology group of this sheaf. Similarly, the second
    spectral sequence gives that
    \[\HH^i(F, \cQ_F^\vee(-2H) = \begin{cases}
        \CC^{561}, &\text{ if }i=4,\\
        0, &\text{ otherwise}.
     \end{cases}
    \]
    From Sequence \eqref{eq:resolution IS} we obtain now immediately that \[\HH^i(F, I_S)=0\text{ for } i\leq2.\] The
    sequence \[0\to I_S\to\OO_F\to\OO_S\to0\] and the fact that $h^i(F,\OO_F)$ is $1,0,1,0,1$ for $i=0,\ldots,4$
    respectively give that $h^1(S,\OO_S)=0$. From $450=\chi(S,\OO_S)=1-q+p_g$ we immediately obtain $p_g=449$. As
    $h^{1,0}=h^{0,1}=0$, so are the Betti numbers $b_1=b_3=0$. Since $S$ is connected, $b_0=b_4=1$. Note that
    $\chi_{\mathrm{top}}={\mathrm c}_2(T_S)=\sum (-1)^ib_i=2565$, giving that $b_2=2563$ and hence from the Hodge
    decomposition and Hodge duality that $h^{1,1}=b_2-2h^{2,0}=1665$.

    For $t=1$, the first spectral sequence and Proposition \ref{prop:coh grass} give
    $E_\infty^{-2,4}=E_1^{-2,4}\cong\CC^{36}$ as the only non-zero term. Hence
    \[
    h^2(F, (\Sym^2\cU)(H))=36
    \]
    is the only non-zero cohomology group of this sheaf. The second spectral sequence for $t=1$ gives that
    \[ h^0(F, \cQ_F^\vee(H))=20, \hspace{8pt} h^2(F, \cQ_F^\vee(H))=1 \]
    are the only two non-trivial cohomology groups. 
    
    The resolution of the ideal sheaf twisted by $3H$ 
    \[ 0\to(\Sym^2\cU_F)(H)\to\cQ_F^\vee(H)\to I_S(3H)\to 0\]
    and the computations above give that $h^3(F, I_S(3H))=0$. Kodaira vanishing gives $h^i(F, \OO_S(3H))=0$ for all $i\geq1$ so
    the sequence \[ 0\to I_S(3H)\to\OO_F(3H)\to \OO_S(3H)\to 0 \] induces $h^2(S, \OO_S(3H))=h^3(F, I_S(3H))$. If $K_S$
    and $3H$ were linearly equivalent and not just equal in the group $\HH^2(S,\QQ)$, then $1=h^2(S, K_S)=h^3(F,
    I_S(3H))$ which is a contradiction to the computation above giving $h^3(F, I_S(3H))=0$. 
    
    Since $q=h^1(S,\OO_S)=0$ is the tangent space to the abelian variety $\Pic^0 S$, this must be zero,
    giving $\Pic S= \mathrm{NS}(S)$. Since $3H$ and $K_S$ are cohomologically but not linearly equivalent, there must be
    torsion in cohomology, or in other words $\Pic^\tau S\neq0$.
\end{proof}

\begin{remark}\label{rem:2-torsion}
    In \cite[Remark 6.4.10]{huybrechts}, it is shown that there is a degree 2 \'etale cover of $S$ trivialising the
    above torsion element, which is, from \eqref{seq:huybrechts}, the difference $K_S-3H_S\in\Pic S$. This cover can be
    realised as the surface in $\PP(\cU_S)$ parametrising the two distinct ramification points of the Gauss map when
    restricted to a line.
\end{remark}

\section{The surface $V$ and its invariants}\label{sec:V}

Let $X\subset\PP^5$ be a smooth cubic, and denote by $V\subset F:=F(X)$ the surface of triple lines, i.e., lines
$\ell\subset X$ so that there exists a 2-plane so that $X\cap\PP^2=3\ell$. Denote also by $\tV\subset \mathrm{Bl}_{S}F$,
the strict transform of $V$. In \cite[4.3-4.4]{gk21} we prove that if $X$ is general, then $V$ is an irreducible surface
and $\tV$ is its smooth normalisation, and we prove that the class of $V$ in the cohomology of $F$ is
$21\rm{c}_2(\cU_F)$. In this section we will give a different geometric interpretation of $\tV$ than the one in
\cite{gk21} and use this to compute the class of $V$ again and some of the invariants of $\tV$. After setting up the
geometric construction, we will perform the computations using Macaulay2 as they are similar to the ones in previous
sections.

We will need the following construction, suggested to us by Kuznetsov. Let $\Fl=\Fl(2,3;6)$ be the 11-dimensional Flag variety
parametrising tuples $(\ell,\Pi)\in\G(2,6)\times\G(3,6)$ so that $\ell\subset\Pi$, and let $\cU_2\subset \cU_3$ be the
universal bundles on $\Fl$ and $L$ the kernel of the surjection $\cU_3^\vee\to \cU_2^\vee$. Denote by $E$ the rank 9
quotient of the following natural inclusion
\begin{align}\label{eq:defE}
0\to 3L\to \Sym^3\cU_3^\vee \to E\to 0,
\end{align}
which is a vector bundle as the iclusion of $3L$ is of full rank at every point.
The equation of the cubic $X$ induces a section $t:\OO_{\Fl}\to\Sym^3\cU_3^\vee$, and hence a section $s:\OO_{\Fl}\to
E$. Denote by $V(s)\subset \Fl$ the vanishing locus of this section. Note that 
\begin{align}\label{eq:sections}
\HH^0(\Fl,\Sym^3\cU_3^\vee)=\HH^0(\G(3,6),\Sym^3\cU_3^\vee)=\HH^0(\PP(\cU_3), \OO_{\PP(\cU_3)}(3))
\end{align}
by the usual Leray argument (in the latter two groups $\cU_3$ is now considered as the universal bundle on $\G(3,6)$), and these
vector spaces also agree with the 56-dimensional $\HH^0(\PP^5, \OO_{\PP^5}(3))$ since the pullback of $\OO_{\PP^5}(1)$ to
the universal family $\PP(\cU_3)$ is $\OO_{\PP(\cU_3)}(1)$.  As $\Sym^3\cU_3^\vee$ is globally
generated, so is $E$, so $V(s)$ has dimension 2 and a general section of $\OO_{\PP^5}(3)$ induces a section of $E$ whose
zero locus is generically reduced (see \cite[Lemma 5.2]{3264}).

Note that the set $V(s)\subset\Fl$ consists of pairs $(\ell,\Pi)$ so that $X\cap\Pi=3\ell$ or $\Pi\subset X$. To
see this, note that if $(\ell,\Pi)$ is already a zero of $t$ then the equation of $X$ vanishes on $\Pi$ from Equation
\eqref{eq:sections}.  For the remaining zeros of $s$, note that $L$ parametrises linear forms on $\cU_3$ which vanish on
$\cU_2$, so that from Sequence \eqref{eq:defE} such a point is an $(\ell,\Pi)$ so that $X\cap\Pi=3\ell$.

If $X$ is a general cubic, then $S$ is smooth and the blowup of $F$ at $S$ parametrises planes tangent to lines in $X$
as it is known (see \cite[Remark 2.2.19]{huybrechts}) that it is isomorphic to the incidence variety 
\[
\Bl_S(F)\cong\{(\ell,\Pi) : \Pi\cap X=2\ell+\ell'\} \subset \Fl \subset \G(2,6)\times\G(3,6).
\]
Under the genericity assumption, $X$ does not contain any $\PP^2$'s and $V(s)$ is necessarily reduced, so the discussion
above gives.
    
\begin{proposition}\label{kuznetsov}
If $X$ is a general cubic, then $V(s)$ is isomorphic to $\tilde{V}$.
\end{proposition}

We give now another proof of the following fact, using the above construction, that was obtained by a different
geometric construction in \cite[Theorem 4.7]{gk21}.

\begin{lemma}
    The class of $V$ in the cohomology of $F$ is given by \[[V]=21\mathrm{c}_2(\cU_F).\]
\end{lemma}
\begin{proof}
    This can be obtained as a consequence of the construction of Proposition \ref{kuznetsov}, and as it involves
    Schubert calculus computations very similar to the ones of sections above, we perform it directly in Macaulay2 in
    the following code, which sets up $E, \cU_2, \cU_3$ etc, computes the class of $V(s)$ in $\Fl$ as the top Chern class
    of $E$, pushes it forward to the Grassmannian $\G(2,6)$, and then compares it with $21\mathrm{c}_2(\cU_F)$:
    \small{
    \begin{verbatim}
loadPackage "Schubert2";
G=flagBundle({2,4}); (U,Q)=G.Bundles; c1=chern_1 U; c2=chern_2 U;
F=flagBundle({2,1,3},6); U2=(F.SubBundles)_1; U3=(F.SubBundles)_2;
E=(symmetricPower(3, dual U3)) - (symmetricPower(3, dual (U3-U2)));
c2UF=((18*c1^2*c2+9*c2^2)*c2);
(map(G, F))_*(chern(9,E))==21*c2UF
    \end{verbatim}
    \vspace*{-\dimexpr2\baselineskip + \topsep + \partopsep}
     \qedhere
    }
\end{proof}

Note that as $V(s)$ is the vanishing of a section of the vector bundle $E$, its ideal sheaf has a Koszul resolution
\[0\to\wedge^9E^\vee\to\ldots\to\wedge^2E^\vee\to E^\vee\to\OO_{\Fl}\to\OO_{V(s)}\to0.\]
Computing using Grothendieck--Riemann--Roch and Schubert calculus we obtain that
\[\chi(\OO_{\tV})=1071,\]
e.g., via the following Macaulay2 code
\begin{verbatim}
sum(10, i -> (-1)^i*(chi exteriorPower(i, dual E)))
\end{verbatim}
On the other hand, as the normal bundle of $\tV$ in $\Fl$ is given by $E|_{\tV}$ (as $\tV$ and $V(s)$ are isomorphic), we
can compute that $K_{\tilde{V}}=3H$, for $H$ the pullback of the Pl\"ucker polarisation restricted to $V\subset F\subset
\G(2,6)$, using
\begin{verbatim}
KtV=chern_1 (cotangentBundle F) + chern_1 E
\end{verbatim}
which we also computed differently in \cite[Proposition 4.6]{gk21} by expressing $\tV$ as a section of a rank two bundle
in $\Bl_S(F)$. We can now easily compute $K_{\tV}^2=8505$ as follows
\begin{verbatim}
integral ((chern_9 E)*(KtV)^2)
\end{verbatim}

What remains in terms of the invariants of $\tilde{V}\cong V(s)$ are the geometric genus $p_g$ and the irregularity $q$,
which satisfy $p_g-q=1070$. As $E$ involves indecomposable bundles on the Flag variety, the Borel--Weil--Bott
computations necessary to compute either of these invariants is much more involved. Nevertheless, very recently, Mboro
\cite{mboro} computed that $p_g=1070$ and $q=0$ by computing the Hodge numbers of the Fano scheme of 2-planes in the
cyclic cover cubic 5-fold associated to $X$ and proving this is an \'etale 3-1 cover of $\tilde{V}$, so all the Hodge
number of $\tilde{V}$ are now also known.

\section{A bound on the degree of irrationality of $F$}\label{sec:irr}

We recently proved in \cite{irrgk} that if $Y\subset\PP^4$ is a smooth cubic threefold and $F(Y)$ its Fano surface of
lines, then \textit{the degree of irrationality} $\irr(F(Y))$, i.e., the minimal degree of a dominant rational map
$F(Y)\dashrightarrow\PP^2$, satisfies
\[\irr(F(Y))\leq 6,\]
with equality if $Y$ is general. In this section we extend the construction of a degree 6 map to the Fano scheme of
lines of any smooth cubic hypersurface. Whether this upper bound is optimal for a general hypersurface remains to be
proven.

We recall first the construction in the case of threefolds, and elaborate on the linear system it is induced by.
\begin{lemma}
Let $Y\subset\PP^4$ be a smooth cubic threefold and $F(Y)\subset\Gr(2,5)\subset\PP:=\PP(\wedge^2\CC^5)=\PP^9$ its Fano
surface of lines. For any hyperplane $H\subset\PP^4$ there is a degree $6$ rational map
\[\phi:F\dashrightarrow Y\cap H\] which is the restriction of the rational map $\psi:\PP\dashrightarrow H$ given by the
sublinear-system $V\subset|\cO_\PP(1)|$ of sections corresponding to Schubert cycles $\sigma_1(\Lambda)$ for $\Lambda$ a
hyperplane in $H$. 
\end{lemma}
\begin{proof}
    The map $\psi_{\Gr(2,5)}:\Gr(2,5)\dashrightarrow H$ takes $[\ell]$ and gives $\ell\cap H\in\PP^4$. Consider now a
    $\Lambda\in|\cO_H(1)|$. Its pullback $\psi_{\Gr(2,5)}^*\Lambda$, which corresponds to lines meeting $\Lambda$, is of
    class $\sigma_1$ and so a section of the Pl\"ucker line bundle $\cO_{\Gr(2,5)}(1)$. Observe that this section
    contains all lines contained inside $H$. In other words, if 
    \[V=|\psi_{\Gr(2,5)}^*\cO_H(1)|\subset|\cO_{\Gr(2,5)}(1)|,\]
    then the base locus $\mathrm{Bs}(V)$ is equal to $\Gr(2,H)$. Projecting now from the $\PP^5$ which is the span of
    $\Gr(2,H)$ in $\PP$
    onto $\PP^3$ we obtain the map $\psi$ whose restriction to $\Gr(2,5)$ is $\psi_{\Gr(2,5)}$. The map $\phi$ has
    degree 6 as there are 6 lines through a general point of $Y$.
\end{proof}
\begin{remark}
In particular, $\psi$ is the projection from the $\PP^5\subset\PP$ containing the Pl\"ucker embedding of
$\Gr(2,H)=\Gr(2,4)$.
\end{remark}

\begin{proposition}
Let $X\subset\PP^{n+1}$ be a smooth cubic hypersurface for $n\geq3$ and $F=F(X)\subset\Gr(2,n+2)$ its Fano scheme of
lines. Then \[\irr(F)\leq6.\] More precisely, we have a degree 6 rational map
\[\phi:F\dashrightarrow R\times Y\]
where $Y=X\cap H$, for $H=\PP^n$, is a hyperplane section of $X$ with one node and hence rational and
$R\cong\PP^{n-3}\subset\PP^{n+1}$ is general. The map $\phi$ is the restriction of the map
\[ (\alpha,\beta):\PP:=\PP(\wedge^2\CC^{n+2})=\PP^{\frac{n(n+3)}{2}}\dashrightarrow R\times H\]
where $\beta$ is given by the $n+1$ sections of $\cO_{\PP}(1)$ cutting out the projective space
$\PP(\wedge^2\CC^{n+1})\subset\PP$ containing the Pl\"ucker embedding of $\Gr(2, H)$ and $\alpha$ is given by the space
of sections of $\cO_{\PP}(1)$ which correspond to Schubert cycles 
\[\sigma_1(T) = \{\ell\in\Gr(2,n+2) : \ell\cap\langle T,\Pi\rangle\neq\emptyset\}\]
for some fixed $\Pi\cong\PP^2$ and $T$ runs over all hyperplanes in $R$.
\end{proposition}
\begin{proof}
    Let $Y=X\cap H$ be a hyperplane section with exactly one node. Note that by projecting from the node inside
    $H=\PP^n$, we obtain a birational map $Y\dashrightarrow\PP^{n-1}$.
    
    Fix now $R=\PP^{n-3}$ and $\Pi=\PP^2$ general inside $\PP^{n+1}$. We will construct a degree 6 map
    $\phi:F\dashrightarrow R\times Y$. Consider a general point $[\ell]\in F$. For the following two points
    \begin{eqnarray*}
        p_\ell&=&R\cap\langle \ell, \Pi\rangle\\
        q_\ell&=&\ell\cap Y,
    \end{eqnarray*}
    define now $\phi([\ell])=(p_\ell,q_\ell)$. For any $q\in X$, there is a subvariety $F_q\subset F$ of dimension $n-3$
    parametrising lines $[\ell]\in F$ so that $\ell$ passes through $q$. This variety $F_q$ in fact embeds in the
    original $\PP^{n+1}$ as a complete intersection of type $(1,1,2,3)$. Fix a $(p,q)\in\phi(F)$.
    The lines through $q$ are parametrised by the space $F_q$ we just described. Note now that, the points $[\ell]\in F_q$ so that
    $p=R\cap\langle\ell,\Pi\rangle$ are precisely the six points of the intersection $\langle p, q, \Pi\rangle\cap F_q$.
    In other words $\phi$ has degree six and we can compose with a birational map $R\times Y\dashrightarrow
    \PP^{2(n-2)}$ to obtain a degree six map $F\dashrightarrow\PP^{2(n-2)}$.
\end{proof}


\begin{thebibliography}{BCP20}

\bibitem[Ame11]{amerik}
Ekaterina Amerik.
\newblock {A computation of invariants of a rational self-map}.
\newblock {\em Annales de la facult{\'{e}} des sciences de Toulouse
  Math{\'{e}}matiques}, 18(3):481--493, 2011.

\bibitem[AV08]{amerikvoisin}
Ekaterina Amerik and Claire Voisin.
\newblock Potential density of rational points on the variety of lines of a
  cubic fourfold.
\newblock {\em Duke Math. J.}, 145(2):379--408, 2008.

\bibitem[BD85]{bd}
Arnaud Beauville and Ron Donagi.
\newblock La vari\'{e}t\'{e} des droites d'une hypersurface cubique de
  dimension {$4$}.
\newblock {\em C. R. Acad. Sci. Paris S\'{e}r. I Math.}, 301(14):703--706,
  1985.

\bibitem[BCP20]{bcp}
Lev~A. Borisov, Andrei C\u{a}ld\u{a}raru, and Alexander Perry.
\newblock Intersections of two {G}rassmannians in {$\Bbb{P}^9$}.
\newblock {\em J. Reine Angew. Math.}, 760:133--162, 2020.

\bibitem[CG72]{cg}
C.~Herbert Clemens and Phillip~A. Griffiths.
\newblock The intermediate {J}acobian of the cubic threefold.
\newblock {\em Ann. of Math. (2)}, 95:281--356, 1972.

\bibitem[EH16]{3264}
David Eisenbud and Joe Harris.
\newblock {\em 3264 and all that---a second course in algebraic geometry}.
\newblock Cambridge University Press, Cambridge, 2016.

\bibitem[FH91]{fultonharris}
William Fulton and Joe Harris.
\newblock {\em Representation theory}, volume 129 of {\em Graduate Texts in
  Mathematics}.
\newblock Springer-Verlag, New York, 1991.
\newblock A first course, Readings in Mathematics.

\bibitem[Fri98]{friedman}
Robert Friedman.
\newblock {\em Algebraic surfaces and holomorphic vector bundles}.
\newblock Universitext. Springer-Verlag, New York, 1998.

\bibitem[Ful98]{fulton}
William Fulton.
\newblock {\em Intersection theory}, volume~2 of {\em Ergebnisse der Mathematik
  und ihrer Grenzgebiete. 3. Folge. A Series of Modern Surveys in Mathematics
  [Results in Mathematics and Related Areas. 3rd Series. A Series of Modern
  Surveys in Mathematics]}.
\newblock Springer-Verlag, Berlin, second edition, 1998.

\bibitem[FP98]{fp}
William Fulton and Piotr Pragacz.
\newblock {\em Schubert varieties and degeneracy loci}, volume 1689 of {\em
  Lecture Notes in Mathematics}.
\newblock Springer-Verlag, Berlin, 1998.
\newblock Appendix J by the authors in collaboration with I. Ciocan-Fontanine.

\bibitem[GK19]{irrgk}
Frank Gounelas and Alexis Kouvidakis.
\newblock Measures of irrationality of the {F}ano surface of a cubic threefold.
\newblock {\em Trans. Amer. Math. Soc.}, 371(10):7111--7133, 2019.

\bibitem[GK21]{gk21}
Frank {Gounelas} and Alexis {Kouvidakis}.
\newblock {Geometry of lines on a cubic fourfold}.
\newblock {\em IMRN}, (to appear), 2022.


\bibitem[HT84]{harristu}
J.~Harris and L.~Tu.
\newblock Chern numbers of kernel and cokernel bundles.
\newblock {\em Invent. Math.}, 75(3):467--475, 1984.

\bibitem[Huy23]{huybrechts}
Daniel Huybrechts.
\newblock {\em The geometry of cubic hypersurfaces}, to appear.
\newblock Cambridge University press, 2023.

\bibitem[Laz04]{lazarsfeldv1}
Robert Lazarsfeld.
\newblock {\em Positivity in algebraic geometry. {I}}, volume~48 of {\em
  Ergebnisse der Mathematik und ihrer Grenzgebiete. 3. Folge. A Series of
  Modern Surveys in Mathematics [Results in Mathematics and Related Areas. 3rd
  Series. A Series of Modern Surveys in Mathematics]}.
\newblock Springer-Verlag, Berlin, 2004.
\newblock Classical setting: line bundles and linear series.

\bibitem[{Mbo}23]{mboro}
Ren{\'e} {Mboro}.
\newblock {Remarks on the geometry of the variety of planes of a cubic
  fivefold}.
\newblock {\em arXiv e-prints}, page arXiv:2301.04997, January 2023.

\bibitem[Muk06]{mukai}
Shigeru Mukai.
\newblock Polarized {$K3$} surfaces of genus thirteen.
\newblock In {\em Moduli spaces and arithmetic geometry}, volume~45 of {\em
  Adv. Stud. Pure Math.}, pages 315--326. Math. Soc. Japan, Tokyo, 2006.

\bibitem[NO21]{no}
Denis Nesterov and Georg Oberdieck.
\newblock Elliptic curves in hyper-{K}\"{a}hler varieties.
\newblock {\em Int. Math. Res. Not. IMRN}, (4):2962--2990, 2021.

\bibitem[OSY19]{osy}
Georg Oberdieck, Junliang Shen, and Qizheng Yin.
\newblock Rational curves in holomorphic symplectic varieties and
  {G}romov-{W}itten invariants.
\newblock {\em Adv. Math.}, 357:106829, 8, 2019.

\bibitem[Pra88]{pragacz}
Piotr Pragacz.
\newblock Enumerative geometry of degeneracy loci.
\newblock {\em Ann. Sci. \'{E}cole Norm. Sup. (4)}, 21(3):413--454, 1988.

\bibitem[SV16]{svfourier}
Mingmin Shen and Charles Vial.
\newblock The {F}ourier transform for certain hyperk\"{a}hler fourfolds.
\newblock {\em Mem. Amer. Math. Soc.}, 240(1139):vii+163, 2016.

\bibitem[Voi04]{voisinintrinsic}
Claire Voisin.
\newblock Intrinsic pseudo-volume forms and {$K$}-correspondences.
\newblock In {\em The {F}ano {C}onference}, pages 761--792. Univ. Torino,
  Turin, 2004.

\bibitem[Wey03]{weyman}
Jerzy Weyman.
\newblock {\em Cohomology of vector bundles and syzygies}, volume 149 of {\em
  Cambridge Tracts in Mathematics}.
\newblock Cambridge University Press, Cambridge, 2003.

\end{thebibliography}
\end{document}